\documentclass[a4paper,11pt]{article}
\usepackage[utf8]{inputenc}
\usepackage[T1]{fontenc}
\usepackage{amsmath, amsthm, amssymb}
\usepackage{physics} % abs command
\usepackage{braket} % macros for sets notation {|}
\usepackage{tikz-cd} % commutative diagramm
\usepackage{faktor} % quotient group
\usepackage{empheq}
\usepackage{graphicx} % includegraphics

\newcommand{\mres}{\mathbin{\vrule height 1.6ex depth 0pt width 0.13ex\vrule height 0.13ex depth 0pt width 0.8ex}}

\theoremstyle{plain}
\newtheorem{thm}{Theorem}[section]
\newtheorem{prop}[thm]{Proposition}
\newtheorem{lem}[thm]{Lemma}
\newtheorem*{thm*}{Theorem}
\newtheorem*{prop*}{Proposition}
\newtheorem*{lem*}{Lemma}

\theoremstyle{definition}
\newtheorem{defi}[thm]{Definition}

\newtheorem*{defi*}{Definition}

\theoremstyle{remark}
\newtheorem{rmk}[thm]{Remark}
\newtheorem*{rmk*}{Remark}

\newcommand{\R}{\mathbf{R}}
\newcommand{\Z}{\mathbf{Z}}
\newcommand{\N}{\mathbf{N}}
\newcommand{\spt}{\mathrm{spt}}
\newcommand{\HH}{\mathcal{H}}
\newcommand{\LL}{\mathcal{L}}
\newcommand{\II}{\mathcal{I}}
\newcommand{\IL}{\Lambda}
\title{Solutions of the (free boundary) Reifenberg Plateau problem}
\author{Camille Labourie}
\date{}

\begin{document}

\maketitle

\begin{abstract}%_abstract
    We solve two variants of the Reifenberg problem for all coefficient groups. We carry out the direct method of the calculus of variation and search a solution as a weak limit of a minimizing sequence. This strategy has been introduced by De Lellis, De Philippis, De Rosa, Ghiraldin and Maggi in \cite{I1},\cite{I2},\cite{I4},\cite{I5} and allowed them to solve the Reifenberg problem. We use an analogous strategy proved in \cite{Labourie} which allows to take into account the free boundary. Moreover, we show that the Reifenberg class is closed under weak convergence without restriction on the coefficient group.
\end{abstract}

\tableofcontents

\section{Definitions and main results}
\subsection{Introduction}
We present the Reifenberg approach to the Plateau problem \cite{Rei} (1960). We fix a $(d-1)$-dimensional compact boundary $\Gamma \subset \R^n$ (the \emph{boundary}). The Reifenberg competitors are the compact sets of $\R^n$ which contain and span $\Gamma$ in the sense of algebraic topology. 
\begin{defi*}[Reifenberg competitors]
    Fix a subgroup $L$ of the homology group $H_{d-1}(\Gamma)$. A Reifenberg competitor is a compact subset $E \subset \R^n$ such that $E$ contains $\Gamma$ and the morphism induced by inclusion
    \begin{equation*}
        \begin{tikzcd}
            H_{d-1}(\Gamma) \arrow[r]   &   H_{d-1}(E)
        \end{tikzcd}
    \end{equation*}
    is zero on $L$.
\end{defi*}
Here the homology theory is the \v{C}ech homology theory with a compact abelian coefficient group. Reifenberg minimizes the (spherical) Hausdorff measure of dimension $d$. He searches a minimizer as a limit of a minimizing sequence in Hausdorff distance. The continuity property of the \v{C}ech theory ensures that such limit is a competitor. However, the area is not lower semicontinuous with respect to convergence in Hausdorff distance. One can imagine a minimizing sequence which has more and more dense tentacles so that the limit set is too large. Reifenberg works with a compact abelian coefficient group so as to benefit from the Exactness Axiom and to be able to cut out the tentacles and patch the holes. His construction leads to an alternative minimizing sequence for which the area is lower semicontinuous.

In \cite{A1} (1968), Almgren solves a variant of the Reifenberg problem where the area is the integral of a \emph{(bounded) elliptic integrand} and the coefficient group is a finitely generated abelian group. His competitors only span a single given cycle of the boundary. Almgren relies on the theory of currents, flat chains and integral varifolds instead of the arguments of Reifenberg.

\begin{figure}[b]\label{fig}
    \centering
    \begin{minipage}{0.46\linewidth}
        \centering
        \includegraphics[width=0.8\linewidth]{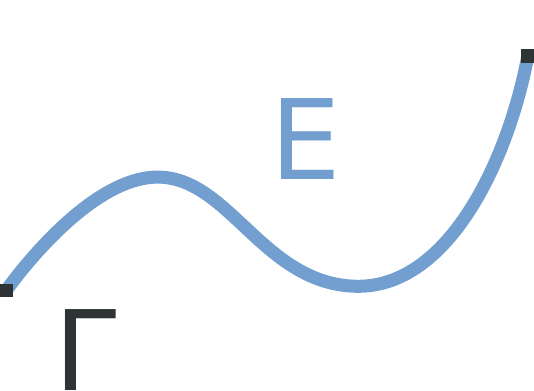}
        \caption{Competitors having a fixed intersection with $\Gamma$.}
    \end{minipage}
    \qquad
    \begin{minipage}{0.46\linewidth}
        \centering
        \includegraphics[width=0.8\linewidth]{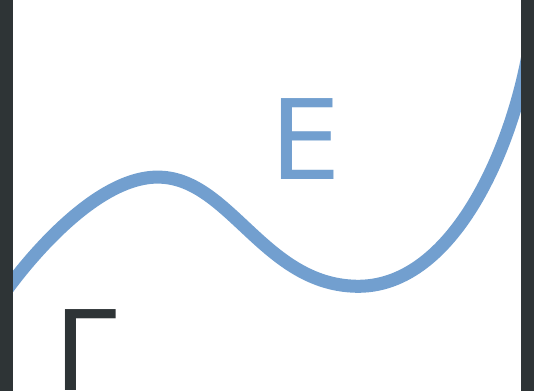}
        \caption{Competitors having a free intersection with $\Gamma$. The intersection $E \cap \Gamma$ may not be $\HH^d$ negligible. One can minimize $\HH^d(E \setminus \Gamma)$ or the whole set $\HH^d(E)$.}
    \end{minipage}
\end{figure}
In \cite{Nakauchi} (1984), Nakauchi solves a free boundary variant of the problem where the area is the Hausdorff measure of dimension $d$ and the coefficient group is a compact abelian group. What we mean by "free boundary" is that $\Gamma$ is not necessarily $(d-1)$-dimensional and that the intersection $E \cap \Gamma$ can vary among the competitors $E$.
\begin{defi*}[Nakauchi competitors]%_nakauchi
    Fix a subgroup $L$ of the homology group $H_{d-1}(\Gamma)$. The Nakauchi competitors are the compacts sets $E \subset \R^n$ such that for all $v \in L$, there exists $u \in H_{d-1}(E \cap \Gamma)$ such that $i_*(u) = v$ and $i'_*(u) = 0$ where $i_*$ and $i'_*$ are the morphisms induced by inclusions:
    \begin{equation*}
        \begin{tikzcd}
            \strut                                                      &   H_{d-1}(\Gamma)\\
            H_{d-1}(E \cap \Gamma) \arrow[ur,"i_*"] \arrow [dr,"i'_*"]\\
            \strut                                                      &   H_{d-1}(E).
        \end{tikzcd}
    \end{equation*}
\end{defi*}
Nakauchi minimizes $\HH^d(E \setminus \Gamma)$ but it would be also interesting to minimize the whole set $\HH^d(E)$ so as take into account the free part $E \cap \Gamma$.

In \cite{DePauw} (2007), De Pauw solves simultaneously the Reifenberg and the Federer--Fleming problems in $\R^3$. In the Reifenberg formulation, the area is the Hausdorff measure of dimension $2$ and the coefficient group is $\Z$. De Pauw proves a natural equivalence between the homology of integral currents and the \v{C}ech homology on locally connected sets. He then builds a minimizing sequence of the Federer--Fleming problem whose supports are a minimizing sequence of the Reifenberg problem. This approach allows to exploit the properties of the currents sequence.

In \cite{Fang} (2015), Fang gives a new solution to the Reifenberg and Nakauchi problems (the free part $E \cap \Gamma$ is not taken into account). The area is the integral of an elliptic integrand and the coefficient group is an arbitrary abelian group. His proof takes advantage of the lower semicontinuity of the area on quasiminimal sets. Thanks to a construction of Feuvrier \cite{Feuv}), Fang obtains an alternative minimizing sequence composed of such sets. In \cite{FK} (2018), Fang and Kolasiński obtain the same existence result by following Almgren’s original ideas.

In \cite{I1} and \cite{I2} (2014), De Lellis, De Philippis, De Rosa, Ghiraldin and Maggi introduce a new direct method based on the \emph{weak convergence} of minimizing sequences in $\R^n \setminus \Gamma$. The area is given by the Hausdorff measure of dimension $d$ in the two first articles, and then by the integral of an elliptic integrand in \cite{I4}, \cite{I5}. The authors also apply this technique to solve the problem of Reifenberg but $E \cap \Gamma$ is not taken into account and a compact abelian coefficient group is required. In \cite{Labourie}, this notion of weak limit is generalized to quasiminimizing sequences in any ambiant space (even containing $\Gamma$). Thus, we can follow this strategy and take into account the free boundary.

\subsection{Reifenberg competitors}
\textbf{Notation} We fix an integer $n \geq 1$, an integer $d$ such that $1 \leq d \leq n$ and a closed set $\Gamma$ of $\R^n$ (the boundary). Given a topological space $S$, an integer $k$ and an abelian group $G$, we denote by $H_k(S;G)$ the $k$\textsuperscript{th} \v{C}ech homology group of $S$ over $G$. We abbreviate this notation as $H_k(S)$ because the coefficient group is not significant in this paper. We fix a subgroup $L$ of $H_{d-1}(\Gamma)$.

We work with a definition of Reifenberg competitors which allows free boundaries as in Fig. \ref{fig}.
\begin{defi}[Reifenberg competitor]\label{defi_reifenberg}%_defi_reifenberg
    A Reifenberg competitor is a compact subset $E \subset \R^n$ such that the morphism induced by inclusion
    \begin{equation}
        \begin{tikzcd}
            H_{d-1}(\Gamma) \arrow[r]   &   H_{d-1}(E \cup \Gamma)
        \end{tikzcd}
    \end{equation}
    is zero on $L$.
\end{defi}

Our goal is to prove the two following results. We omit the regularity of the boundary here (see Theorems \ref{rei_sol1} and \ref{rei_sol2} for the full statement). \emph{Admissible energies} are precised in Definition \ref{defi_energy}.
\begin{thm*}[with the free boundary]%_rei_sol1
    Let $\II$ be an admissible energy. We assume that 
    \begin{equation}
        m := \inf \set{\II(E) | E \ \text{Reifenberg competitor}} < \infty
    \end{equation}
    and that there exists a compact set $C \subset \R^n$ such that
    \begin{equation}
        m = \inf \set{\II(E) | E \ \text{Reifenberg competitor},\ E \subset C}.
    \end{equation}
    Then there exists a Reifenberg competitor $E \subset C$ such that $\II(E) = m$.
\end{thm*}
The next theorem is equivalent to \cite[Theorem 1.3]{Fang} (which is based on Feuvrier's construction) and \cite[Theorem 3.4]{I5} (which is based on weak convergence of minimizing sequences).
\begin{thm*}[without the free boundary]%_rei_sol2
    Let $\II$ be an admissible energy. We assume that 
    \begin{equation}
        m := \inf \set{\II(E \setminus \Gamma) | E \ \text{Reifenberg competitor}} < \infty
    \end{equation}
    and that there exists a compact set $C \subset \R^n$ such that
    \begin{equation}
        m = \inf \set{\II(E \setminus \Gamma) | E \ \text{Reifenberg competitor},\ E \subset C}.
    \end{equation}
    Then there exists a Reifenberg competitor $E \subset C$ such that $\II(E \setminus \Gamma) = m$.
\end{thm*}
\begin{rmk*}
    If $\Gamma$ is compact and $\II(\Gamma) < \infty$, this amounts to minimizing $\II(E)$ among Reifenberg competitors containing $\Gamma$.
\end{rmk*}

In the rest of this subsection, we compare Definition \ref{defi_reifenberg} with others definition of Reifenberg competitors. The next lemma shows that Definition \ref{defi_reifenberg} is equivalent to the original definition of Reifenberg when $\Gamma$ is regular enough and $L \ne 0$. It has been suggested by Ulrich Menne.
\begin{lem}
    Assume that $\Gamma$ is a compact connected $C^1$ manifold of dimension $d-1$ imbedded in $\R^n$. Let a compact subset $E \subset \R^n$ be such that the morphism induced by inclusion
    \begin{equation}
        \begin{tikzcd}
            H_{d-1}(\Gamma) \arrow[r]   &   H_{d-1}(E \cup \Gamma)
        \end{tikzcd}
    \end{equation}
is non injective. Then $\Gamma \subset E$.
\end{lem}
\begin{proof}
    We proceed by contraposition and assume that there exists $x \in \Gamma \setminus E$. Let us introduce some preliminary objects. According to the Triangulation Theorem \cite[Chapter IV, Theorem 12A]{Whitney}, there exists a simplicial complex\footnote{We follow the formalism of \cite[Chapter II]{ES} where a simplicial complex $K$ is a finite collection of faces and $\abs{K}$ is the set of points which belong to simplexes of $K$.} $K$ and a homeomorphic map $\pi\colon \abs{K} \to \Gamma$. Let $p \in \abs{K}$ be the point such that $\pi(p) = x$. We recall that the \emph{star} of p,
    \begin{equation}
        \mathrm{St}(p) := \bigcup \set{\mathrm{int}(\sigma) | \sigma \in K,\ p \in \sigma},
    \end{equation}
    is an open set of $\abs{K}$. Moreover, $L: = \set{\sigma \in K | p \notin \sigma}$ is a subcomplex of $K$ such that $\abs{L} = \abs{K} \setminus \mathrm{St}(p)$. Since $\pi^{-1}(\Gamma \setminus E)$ is an open set containing $p$, we can replace $K$ by its barycentric subdivision until $\overline{\mathrm{St}(p)} \subset \pi^{-1}(\Gamma \setminus E)$. In conclusion, $U := \pi(\mathrm{St}(p))$ is a relative open set of $\Gamma$, $\overline{U} \subset \Gamma \setminus E$ and the pair $(\Gamma, \Gamma \setminus U)$ is triangulated by $(K,L)$.

    Now, we consider the commutative diagram
\begin{equation}
    \begin{tikzcd}
        H_{d-1}(\Gamma) \arrow[r] \arrow[d] &   H_{d-1}(\Gamma, \Gamma \setminus U) \arrow[d]\\
        H_{d-1}(E \cup \Gamma) \arrow[r]    &   H_{d-1}(E \cup \Gamma, (E \cup \Gamma) \setminus U)
    \end{tikzcd}
\end{equation}
where the arrows are induced by inclusions. We are going to show that 
\begin{equation}
    \begin{tikzcd}
        H_{d-1}(\Gamma) \arrow[r]   &   H_{d-1}(\Gamma, \Gamma \setminus U)
    \end{tikzcd}
\end{equation}
is injective and then that
\begin{equation}
    \begin{tikzcd}
        H_{d-1}(\Gamma, \Gamma \setminus U) \arrow[r]   &   H_{d-1}(E \cup \Gamma, (E \cup \Gamma) \setminus U)
    \end{tikzcd}
\end{equation}
is an isomorphism. It will follow that
\begin{equation}
    \begin{tikzcd}
        H_{d-1}(\Gamma) \arrow[r]   &   H_{d-1}(E \cup \Gamma)
    \end{tikzcd}
\end{equation}
is injective.

The set $\Gamma$ is a connected manifold so, according to \cite[Theorem 3.26]{Hatcher}, the morphism
\begin{equation}
    \begin{tikzcd}
        H_{d-1}(\Gamma) \arrow[r]   &   H_{d-1}(\Gamma, \Gamma \setminus x)
    \end{tikzcd}
\end{equation}
is injective with respect to the singular homology. The factorisation $\Gamma \subset (\Gamma, \Gamma \setminus U) \subset (\Gamma, \Gamma \setminus x)$ implies that
\begin{equation}\label{gamma_injectivity}
    \begin{tikzcd}
        H_{d-1}(\Gamma) \arrow[r]   &   H_{d-1}(\Gamma, \Gamma \setminus U)
    \end{tikzcd}
\end{equation}
is also injective with respect to the singular homology. Let us justify that this result can be passed on to the \v{C}ech homology. According to \cite[Chapter IX, Corollary 9.4]{ES}, the \v{C}ech homology is an (exact) homology theory on the category of triangulable pairs and their simplicial maps. Then, the uniqueness Theorem \cite[Chapter III, Theorem 10.1]{ES} implies that the \v{C}ech homology and the singular homology are isomorphic in a natural way on this category. As the pair $(\Gamma, \Gamma \setminus U)$ is triangulated, we conclude that the morphism (\ref{gamma_injectivity}) is also injective with respect to the \v{C}ech homology.

The second part relies on the Excision axiom and the fact that $\overline{U} \subset \Gamma \setminus E$. The covering
\begin{equation}
    E \cup \Gamma = (E \cup \Gamma \setminus \overline{U}) \cup (\Gamma \setminus E)
\end{equation}
is open in $E \cup \Gamma$ so $E \cup \Gamma$ is covered by the relative interiors of $(E \cup \Gamma) \setminus U$ and $\Gamma$. We can apply the Excision axiom \cite[Chapter IX, Theorem 6.1]{ES} to see that
\begin{equation}
    \begin{tikzcd}
        H_{d-1}(\Gamma, \Gamma \setminus U) \arrow[r]   &   H_{d-1}(E \cup \Gamma, (E \cup \Gamma) \setminus U)
    \end{tikzcd}
\end{equation}
is an isomorphism.
\end{proof}

Next, we compare Definition \ref{defi_reifenberg} to the definition of Nakauchi seen in introduction (or \cite[Definition 1]{Nakauchi}). Let $E$ be a compact subset of $\R^n$ and consider the commutative diagram
\begin{equation*}
    \begin{tikzcd}
        \strut                                                      &   H_{d-1}(\Gamma) \arrow[dr,"j_*"]    &   \\
        H_{d-1}(E \cap \Gamma) \arrow[ur,"i_*"] \arrow [dr,"i'_*"]  &                                       &   H_{d-1}(E \cup \Gamma).\\
        \strut                                                      &   H_{d-1}(E) \arrow[ur,"j'_*"]        &   
    \end{tikzcd}
\end{equation*}
Observe that $E$ satisfies Definition \ref{defi_reifenberg} if and only if all elements of the form $(v,0) \in L \otimes H_{d-1}(E)$ are in the kernel of $j_* - j'_*$. And $E$ is a Nakauchi competitor if and only if all elements of the form $(v,0) \in L \otimes H_{d-1}(E)$ are in the image of $(i_*,i'_*)$. Assuming that the Mayer Vietoris sequence holds for the sets $\Gamma$, $E$ in $E \cup \Gamma$, the following sequence is exact:
\begin{equation*}
    \begin{tikzcd}
        H_{d-1}(E \cap \Gamma) \arrow[r,"{(i_*,i'_*)}"] &   H_{d-1}(\Gamma) \otimes H_{d-1}(E) \arrow[r,"j_* - j'_*"]   &   H_{d-1}(E \cup \Gamma).
    \end{tikzcd}
\end{equation*}
Thus, the Mayer Vietoris sequence implies that Definition \ref{defi_reifenberg} is equivalent to the definition of Nakauchi. In that sense, we consider these definitions to be essentially equivalent. We favor Definition \ref{defi_reifenberg} because we can prove that it is closed under weak convergence without restriction on the coefficient group (see Lemma \ref{rei_limit}).

\subsection{Deformations and energies}
\textbf{Notation} Here the ambiant space is an open set $X$ of $\R^n$ and $\Gamma$ is a relatively closed set of $X$. The interval $[0,1]$ is denoted by the capital letter $I$. Given a set $E \subset X$ and a function $F\colon I \times E \to X$, the notation $F_t$ means $F(t,\cdot)$. Given two sets $A, B \subset \R^n$, the notation $A \subset \subset B$ means that there exists a compact set $K \subset \R^n$ such that $A \subset K \subset B$.%_notations

Sliding deformations have been introduced by David \cite[Definition 1.3]{Sliding} to study the \emph{restricted sets} of Almgren with an additional boundary constraint.
\begin{defi}[Sliding deformation along a boundary]%_defi_sliding
    Let $E$ be a closed, $\HH^d$ locally finite subset of $X$. A \emph{sliding deformation} of $E$ in an open set $U \subset X$ is a Lipschitz map $f\colon E \to X$ such that there exists a continuous homotopy $F\colon I \times E \to X$ satisfying the following conditions:
    \begin{subequations}
        \begin{align}
            &   F_0 = \mathrm{id}\\
            &   F_1 = f\\
            &   \forall t \in I,\ F_t(E \cap \Gamma) \subset \Gamma\\
            &   \forall t \in I,\ F_t(E \cap U) \subset U\\
            &   \forall t \in I,\ F_t = \mathrm{id} \ \text{in} \ E \setminus K,
        \end{align}
    \end{subequations}
    where $K$ is some compact subset of $E \cap U$. Alternatively, the last axiom can be stated as
    \begin{equation}
        \Set{x \in E | \exists \, t \in I, \ F_t(x) \ne x} \subset \subset E \cap U.
    \end{equation}
\end{defi}

The next definition comes from \cite[Definition 25.3 and Remark 25.87]{Sliding}. It is a slight generalization of \cite[Definition 1.6(2)]{A1} and \cite[Definition IV.1(7)]{A2}. After the statement, we will give a few explanations and compare it to other definitions.
\begin{defi}[Admissible energy]\label{defi_energy}%_defi_energy
    An \emph{admissible energy} in $X$ is a Borel regular measure $\II$ in $X$ which satisfies the following axioms:
\begin{enumerate}
    \item There exists $\IL \geq 1$ such that $\IL^{-1} \HH^d \leq \II \leq \IL \HH^d$.
    \item Let $x \in X$, let a $d$-plane $V$ passing through $x$, let a $C^1$ map $f\colon V \to \mathbf{R}^n$ be such that $f(x) = x$ and $Df(x)$ is the inclusion map $\overrightarrow{V} \hookrightarrow \R^n$. Then
    \begin{equation}
        \lim_{r \to 0} \frac{\II(f(V \cap B(x,r)))}{\II(V \cap B(x,r))} = 1.
    \end{equation}
\item For each $x \in X$, there exists $R > 0$ and $\varepsilon \colon \mathopen{]}0,R\mathclose{]} \to \R^+$ such that $\overline{B}(x,R) \subset X$, $\lim_{r \to 0} \varepsilon(r) = 0$ and
    \begin{equation}\label{ellipticity_david}
        \II(V \cap B(x,r)) \leq \II(S \cap B(x,r)) + \varepsilon(r)r^d
    \end{equation}
    whenever $0 < r \leq R$, $V$ is a $d$-plane passing through $x$ and $S \subset \overline{B}(x,r)$ is a compact $\HH^d$ finite set which cannot be mapped into $V \cap \partial B(x,r)$ by a Lipschitz mapping $\psi\colon \overline{B}(x,r) \to \overline{B}(x,r)$ such that $\psi = \mathrm{id}$ on $V \cap \partial B(x,r)$.
    \end{enumerate}
\end{defi}
The third axiom is important to establish the lower semicontinuity. Let us say that $(E_k)$ is a minimizing sequence of competitors such that $\II \mres E_k \to \mu$ where $\mu$ is a $d$-rectifiable Radon measure. The third axiom is the main argument to show that
\begin{equation}
    \II \mres E_\infty \leq \mu,
\end{equation}
where $E_\infty = \spt(\mu)$. This yields in particular $\II(E_\infty) \leq \lim_k \II(E_k)$. The reader can find an example below Definition 25.3 in \cite{Sliding} where $\II(E_\infty)$ is too large.

Here is an example of admissible energy. We consider two Borel functions,
\begin{equation}
    i\colon X \times G(d,n) \to \mathopen{]}0,\infty\mathclose{[} \ \text{and} \ j\colon X \to \mathopen{]}0,\infty\mathclose{[}
\end{equation}
called the \emph{integrands} and we define the corresponding energy by the formula
\begin{equation}
    \II(S) = \int_{S_r} i(y,T_yS) \, \mathrm{d}\HH^d(y) + \int_{S_u} j(y) \, \mathrm{d}\HH^d(y)
\end{equation}
where $S \subset X$ is a Borel $\HH^d$ finite set and $S_r$, $S_u$ are its $d$-rectifiable and purely $d$-unrectifiable parts. We assume $\IL^{-1} \leq i,j \leq \IL$ so that $\IL^{-1} \HH^d \leq \II \leq \IL \HH^d$ on such sets $S$. We define of course $\II(S) = \infty$ on Borel sets $S \subset X$ of infinite $\HH^d$ measure and we extend $\II$ on all subsets of $X$ by Borel regularity. Next, we assume that $i$ is continuous on $X \times G(d,n)$ and we show that this implies the second axiom. Let $x$, $V$ and $f$ be as in (ii). To shorten a bit the notations, we assume $x = 0$ and we denote $i(0,V)$ by $i_0$. For $r > 0$, we denote $B(0,r)$ by $B_r$ and $f(V \cap B(0,r))$ by $S_r$. Fix $\varepsilon > 0$. According to the inverse function theorem, there exists $R > 0$ such that $f$ induces a $(1+\varepsilon)$-bilipschitz diffeomorphism from $V \cap B_R$ to $S = S_R$. Note that the function $y \mapsto (y,T_y S)$ is continuous on $S$ and that at $y = 0$, $(y,T_y S) = (0,V)$. As $i$ is continuous, there exists $R' > 0$ such that for all $y \in S \cap B_{R'}$
\begin{equation}
    (1 + \varepsilon)^{-1} i_0 \leq i(y,T_y S) \leq (1 + \varepsilon) i_0.
\end{equation}
If $r > 0$ is small enough so that $B_r \subset B_R \cap f^{-1}(B_{R'})$, we have $S_r \subset S \cap B_{R'}$ and thus
\begin{equation}
    (1 + \varepsilon)^{-1} i_0 \leq \frac{\II(S_r)}{\HH^d(S_r)} \leq (1 + \varepsilon) i_0.
\end{equation}
The function $f$ is $(1+\varepsilon)$-bilipschitz on $V \cap B_r$ so this simplifies to 
\begin{equation}
    (1 + \varepsilon)^{d-1} i_0 \leq \frac{\II(S_r)}{\HH^d(V \cap B_r)} \leq (1 + \varepsilon)^{d+1} i_0.
\end{equation}
We deduce that
\begin{equation}
    \lim_{r \to 0} \frac{\II(S_r)}{\HH^d(V \cap B_r)} = i_0
\end{equation}
and by the same reasoning,
\begin{equation}
    \lim_{r \to 0} \frac{\II(V \cap B_r)}{\HH^d(V \cap B_r)} = i_0.
\end{equation}
The second axiom follows. It is more difficult to construct integrands that yield the third axiom. A first example is when $\II = \HH^d$ or when $i$ depends on $x$ alone with $i = j$ (we will detail this later). Let us try to relate our integrands to other definitions. The energy introduced by Almgren in \cite[Definition 1.6(2)]{A1} is like above with $X = \R^n$ and $i$ of class $C^3$ but (iii) is replaced by a stronger condition called \emph{ellipticity bound}. We present this condition in more detail. For $x \in \R^n$, let
\begin{equation}
    \II^x(S) = \int_{S_r} i(x,T_yS) \, \mathrm{d}\HH^d(y) + j(x) \HH^d(S_u)
\end{equation}
be defined on Borel $\HH^d$ finite sets $S \subset \R^n$. Almgren requires that there exists a continuous function $c\colon \R^n \to ]0,\infty[$ such that for all $x \in \R^n$,
\begin{multline}\label{ellipticity_almgren}
    \II^x(S \cap B(x,r)) - \II^x(V \cap B(x,r)) \\\geq c(x) \left[\HH^d(S \cap B(x,r)) - \HH^d(V \cap B(x,r))\right]
\end{multline}
whenever $r > 0$, $V$ is a $d$-plane passing through $x$ and $S \subset \overline{B}(x,r)$ is a compact, $\HH^d$ rectifiable, $\HH^d$ finite set which spans $V \cap \partial B(x,r)$. In \cite[Definition IV.1(7)]{A2}, Almgren restricts this definition to the functions $c$ which are positive constants. The energy introduced by De Lellis, De Philippis, De Rosa, Ghiraldin and Maggi in \cite[Definitions 1.3, 1.5 and Remark 1.8]{I5} is also like above with $X = \R^n$ and $i$ of class $C^1$ but (iii) is replaced by a new axiom called \emph{atomic condition}. This axiom characterizes the integrands for which Allard's rectifiability theorem holds (\cite{I3}). In codimension one, it is equivalent to a convexity condition on the integrand. In the general case, \cite{DRK} shows that the atomic condition implies an ellipticity condition of the form
\begin{multline}\label{ellipticity_weak}
    \left[\HH^d(S \cap B(x,r)) - \HH^d(V \cap B(x,r))\right] > 0 \\\implies \left[\II^x(S \cap B(x,r)) - \II^x(V \cap B(x,r))\right] > 0
\end{multline}
whenever $r > 0$, $V$ is a $d$-plane passing through $x$ and $S \subset \overline{B}(x,r)$ is a compact, $\HH^d$ finite set which spans $V \cap \partial B(x,r)$.

In this paragraph, we compare (iii) to Almgren’s ellipticity bound. We allow $d$-dimensional sets $S$ which have a purely unrectifiable part because we allow competitors which have a purely unrectifiable part in the direct method (Proposition \ref{cor_direct}). Almgren only tests the ellipticity bound against rectifiable sets but \cite[Corollary 5.13]{DRK} justifies that this is not a loss of generality. Next, we show that an inequality such as (\ref{ellipticity_almgren}) implies (\ref{ellipticity_david}). We work at the point $x = 0$ and we fix a constant $c > 0$. Let $r > 0$ (to be chosen small enough), let $V$ be a $d$-plane passing through $0$, let $S \subset \overline{B}(0,r)$ be a compact $\HH^d$ finite set such that $V \cap \overline{B}(0,r) \subset p_V(S)$ and let us assume that
\begin{multline}\label{S_assumptions}
    \II^0(S \cap B(0,r)) - \II^0(V \cap B(0,r)) \\\geq c \left[\HH^d(S \cap B(0,r)) - \HH^d(V \cap B(0,r))\right].
\end{multline}
Note that $\HH^d(V \cap B(0,r)) \leq \HH^d(S \cap B(0,r))$ because $V \cap \overline{B}(0,r) \subset p_V(S)$ and $p_V$ is $1$-Lipschitz. We define
\begin{multline}
    \varepsilon(r) = \sup\Set{\abs{i(y,W) - i (0,W)} | y \in B(0,r),\ W \in G(d,n)}\\+ \sup\Set{\abs{j(y) - j(0)} | y \in B(0,r)}.
\end{multline}
The function $i$ is uniformly continuous on the compact set $\overline{B}(0,1) \times G(d,n)$ and $j$ as well on $\overline{B}(0,1)$ so $\lim_{r \to 0} \varepsilon(r) = 0$. We assume $r$ small enough so that $\varepsilon(r) \leq c$. According to the definition of $\varepsilon(r)$,
\begin{equation}
    \abs{\II(S \cap B(0,r)) - \II^0(S \cap B(0,r))} \leq \varepsilon(r) \HH^d(S \cap B(0,r))
\end{equation}
and similarly
\begin{equation}
    \abs{\II(V \cap B(0,r)) - \II^0(V \cap B(0,r))} \leq \varepsilon(r) \HH^d(V \cap B(0,r)).
\end{equation}
Plugging this into (\ref{S_assumptions}), we get
\begin{align}
    \begin{split}
        &\II(S \cap B(0,r)) - \II(V \cap B(0,r))\\
        &\geq c \left[\HH^d(S \cap B(0,r)) - \HH^d(V \cap B(0,r))\right]\\
        &\qquad - \varepsilon(r) \HH^d(S \cap B(0,r)) - \varepsilon(r) \HH^d(V \cap B(0,r))
    \end{split}\\
    \begin{split}
        &\geq (c - \varepsilon(r))\left[\HH^d(S \cap B(0,r)) - \HH^d(V \cap B(0,r))\right]\\
        &\qquad - 2 \varepsilon(r) \HH^d(V \cap B(0,r))
    \end{split}\\
        &\geq - 2 \varepsilon(r) \HH^d(V \cap B(0,r))\\
        &\geq - 2 \varepsilon(r) \omega_d r^d
\end{align}
where $\omega_d$ is the $\HH^d$ measure of the $d$-dimensional unit disk. This is (\ref{ellipticity_david}) as promised. We should not forget to mention that if $i$ depends on $x$ alone with $i = j$, then we have $\II^0 = i(0) \HH^d$ so (\ref{S_assumptions}) is clearly verified.

\section{Operations on Reifenberg competitors}
We come back to Reifenberg competitors and their properties. We present the main operations that preserve these competitors. Unless otherwise indicated, $\Gamma$ is just a closed set of $\R^n$. 
\begin{lem}\label{rei_supset}%_rei_subset
    Let $E$ be a Reifenberg competitor. Let $F$ be a compact subset of $\R^n$ containing $E$. Then $F$ is a Reifenberg competitor.
\end{lem}
\begin{proof}
    This follows from the following commutative diagram
    \begin{equation*}
        \begin{tikzcd}
            H_{d-1}(\Gamma) \arrow[r] \arrow [rd]   &   H_{d-1}(E \cup \Gamma) \arrow [d]\\
                                                    &   H_{d-1}(F \cup \Gamma)
        \end{tikzcd}
    \end{equation*}
    where the arrows correspond to the morphisms induced by inclusion.
\end{proof}

\begin{lem}\label{rei_image}%_rei_image
    Let $E$ be a Reifenberg competitor. Let $f\colon E \cup \Gamma \to \R^n$ be a continuous map such that there exists a continuous map $F\colon I \times \Gamma \to \Gamma$ satisfying $F_0 = \mathrm{id}$ and $F_1 = f$. Then $f(E)$ is a Reifenberg competitor.
\end{lem}
\begin{proof}
    Consider the following commutative diagram
    \begin{equation*}
        \begin{tikzcd}
            H_{d-1}(\Gamma)    \arrow[r] \arrow[d,"f_*"]    &   H_{d-1}(E \cup \Gamma) \arrow[d,"f_*"]\\
            H_{d-1}(\Gamma) \arrow[r]                       &   H_{d-1}(f(E) \cup \Gamma)
        \end{tikzcd}
    \end{equation*}
    where the unlabeled arrows are the morphisms induced by inclusion. As $f\colon \Gamma \to \Gamma$ is homotopic to $\mathrm{id}$, $f_* = \mathrm{id}$ on $H_{d-1}(\Gamma)$.
\end{proof}
This lemma assumed $f$ to be defined on $E \cup \Gamma$ but the image $f(E)$ depends only on the values of $f$ on $E$. The two following remarks complete the lemma by justify that it is generally enough for $f$ to be defined on $E$. The second remark applies to sliding deformations when $\Gamma$ is regular enough.
\begin{rmk}
    Let $f\colon E \to \R^n$ be a continuous map such that $f = \mathrm{id}$ on $E \cap \Gamma$. As $E$ and $\Gamma$ are closed sets of $\R^n$, the gluing
    \begin{equation}
        g =
        \begin{cases}
            f           &   \text{in} \ E\\
            \mathrm{id} &   \text{in} \ \Gamma
        \end{cases}
    \end{equation}
    is continuous. Then $G_t = (1-t) \mathrm{id} + tg$ is a continuous homotopy from $\mathrm{id}$ to $g$ and $G_t = \mathrm{id}$ on $\Gamma$. We deduce that $f(E)$ is a Reifenberg competitor.
\end{rmk}
\begin{rmk}
    Let $f\colon E \to \R^n$ be a continuous map such that there exists a continuous map $F\colon I \times (E \cap \Gamma) \to \Gamma$ satisfying $F_0 = \mathrm{id}$ and $F_1 = f$. Let us assume that $\Gamma$ is a neighborhood retract i.e. there exists an open set $O \subset \R^n$ and a continuous map $r\colon O \to \Gamma$ such that $r = \mathrm{id}$ on $\Gamma$. According to the Homotopy Extension Lemma, $F$ extends as a continuous map $F\colon I \times \Gamma \to \Gamma$. Moreover, the gluing
    \begin{equation}
        g =
        \begin{cases}
            f   &   \text{in} \ E\\
            F_1 &   \text{in} \ \Gamma
        \end{cases}
    \end{equation}
    is continous because $E$ and $\Gamma$ are closed sets of $\R^n$. We deduce that $f(E)$ is a Reifenberg competitor.
\end{rmk}

\begin{lem}\label{rei_limit}%_rei_limit
    Let $(E_k) \subset \R^n$ be a sequence of Reifenberg competitors. Let $E$ be a compact subset of $\R^n$. We assume that
    \begin{enumerate}
        \item there exists a compact set $C \subset \R^n$ such that for all $k$, $E_k \subset C$;
        \item for all open sets $V$ containing $E \cup \Gamma$,
            \begin{equation}
                \lim_k \HH^d(E_k \setminus V) = 0.
            \end{equation}
    \end{enumerate}
    Then $E$ is a Reifenberg competitor.
\end{lem}
The proof requires a preliminary lemma about the general position of spheres. For $x \in \R^n$ and $r > 0$, let $S(x,r)$ denote the Euclidean sphere of center $x$ and radius $r$ of $\R^n$. Given an integer $k$, a \emph{$k$-sphere} is an Euclidean sphere of positive radius relative to a $(k+1)$-affine plane. We extend this definition to $k < 0$, by calling \emph{$k$-sphere} the empty set.
\begin{lem}\label{spheres_positions}
    Let $S^k$ be a $k$-sphere in $\R^n$ and let $x$ be a point in $\R^n$. Then for all $r > 0$ (except for at most one value), $S^k \cap S(x,r)$ is a subset of a $(k-1)$-sphere.
\end{lem}
\begin{proof}
    We assume $k \geq 1$. The proof is based on the observation that the intersection of a $(n-1)$-sphere with a $k$-affine plane is either empty, a point, or a $(k-1)$-sphere. In all cases, this intersection is part of a $(k-1)$-sphere. Let $P_0$ be the $(k+1)$-affine plane associated to $S^k$, let $x_0 \in P_0$ be the center of $S^k$ and $r_0 > 0$ be its radius so $S^k = P_0 \cap S(x_0,r_0)$. For $r > 0$, a point $y \in S^k \cap S(x,r)$ is characterized by the system
    \begin{subequations}
        \begin{empheq}[left=\empheqlbrace]{align}
            &   y \in P_0\\
            &   \abs{y - x} = r\\
            &   \abs{y - x_0} = r_0
        \end{empheq}
    \end{subequations}
    or equivalently
    \begin{subequations}
        \begin{empheq}[left=\empheqlbrace]{align}
            &   y \in P_0\label{sphere1}\\
            &   \abs{y - x} = r\label{sphere3}\\
            &   \abs{y - x}^2 - \abs{y - x_0}^2 = r^2 - r_0^2\label{sphere2}.
        \end{empheq}
    \end{subequations}
    Assume $x = x_0$. If $r \ne r_0$ (this excludes one value of $r$), equation (\ref{sphere2}) has no solutions so $S^k \cap S(x,r)$ is empty so is part of a $(k-1)$-sphere. Assume $x \ne x_0$. Equation (\ref{sphere2}) defines an hyperplane and, if $\abs{x - x_0}^2 \ne r^2 - r_0^2$ (this excludes at most one value of $r$), this hyperplane does not contain $x_0$. Then, the intersection of the two planes (\ref{sphere1}) and (\ref{sphere2}) is included in a $k$-affine plane. The intersection of this plane with the sphere (\ref{sphere3}) is part of a $(k-1)$-sphere as seen in introduction.
\end{proof}

The following proof makes use of the notion of complex (as in \cite[Subsection 2.1]{Labourie}) and the Federer--Fleming projection (as in \cite[Propositon 3.1]{DS} or \cite[Lemma 2.9]{Labourie}).
\begin{proof}[Proof of Lemma \ref{rei_limit}]
    By Lemma \ref{rei_supset}, the sequence $(E_k \cup E)_k$ also satisfies the lemma assumptions. So without loss of generality, we assume that for all $k$, $E \subset E_k$. We also assume that $\Gamma$ is non-empty so that the coverings considered in the proof are also non-empty. We define a \emph{general covering} as an open family $\gamma = (\gamma_j)_{j \in V_\gamma}$ of $\R^n$ satisfying the following properties:
    \begin{enumerate}
        \item there exists $k$ such that $E_k \cup \Gamma \subset \bigcup_{j \in V_\gamma} \gamma_j$;
        \item for every subset $S \subset V_\gamma$ of cardinal $d+1$,
            \begin{equation}
                \bigcap_S \gamma_j \ne \emptyset \implies (E \cup \Gamma) \cap \bigcap_S \gamma_j \ne \emptyset.
            \end{equation}
    \end{enumerate}
    The main goal of the proof is to show that for any open covering $\alpha = (\alpha_i)_i$ of $E \cup \Gamma$, there exists a general covering $\gamma = (\gamma_j)_{j \in V_\gamma}$ such that $((E \cup \Gamma)\cap \gamma_j)_j$ is a refinement of $\alpha$. Let us explain how to conclude from there. Let $\gamma$ be a general covering and let $k$ be an index such that $E_k \cup \Gamma \subset \bigcup_{j \in V_\gamma} \gamma_j$. The inclusions $\Gamma \subset E \cup \Gamma \subset E_k \cup \Gamma$ induce morphisms $i_*$ and $j_*$:
    \begin{equation*}
        \begin{tikzcd}
            H_{d-1}(\Gamma) \arrow[r,"i_*"]    &   H_{d-1}(E \cup \Gamma) \arrow[r,"j_*"] &   H_{d-1}(E_k \cup \Gamma).
        \end{tikzcd}
    \end{equation*}
    The covering $\gamma$ induces simplicial complexes
    \begin{subequations}\label{simplicial_complexes}
    \begin{align}
        K(\Gamma)           &   = \set{S \subset V_\gamma \ \text{finite} | \Gamma \cap \bigcap_S \gamma_j \ne \emptyset},\\
        K(E \cup \Gamma)    &   = \set{S \subset V_\gamma \ \text{finite} | (E \cup \Gamma) \cap \bigcap_S \gamma_j \ne \emptyset},\\
        K(E_k \cup \Gamma)  &   = \set{S \subset V_\gamma \ \text{finite} | (E_k \cup \Gamma) \cap \bigcap_S \gamma_j \ne \emptyset}
    \end{align}
\end{subequations}
    and the inclusions $K(\Gamma) \subset K(E \cup \Gamma) \subset K(E_k \cup \Gamma)$ induce morphisms $i_{\gamma*}$ and $j_{\gamma*}$:
    \begin{equation*}
        \begin{tikzcd}
            H_{d-1}(K(\Gamma)) \arrow[r,"i_{\gamma*}"]    &   H_{d-1}(K(E \cup \Gamma)) \arrow[r,"j_{\gamma*}"] &   H_{d-1}(K(E_k \cup \Gamma)).
        \end{tikzcd}
    \end{equation*}
    Finally, there exists projection morphisms $\pi_\Gamma$, $\pi$, $\pi_k$ such that the following diagramm commutes
    \begin{equation*}
        \begin{tikzcd}
            H_{d-1}(\Gamma) \arrow[r,"i_*"] \arrow[d,"\pi_\Gamma"]  &   H_{d-1}(E \cup \Gamma) \arrow[r,"j_*"] \arrow[d,"\pi"]              &   H_{d-1}(E_k \cup \Gamma) \arrow[d,"\pi_k"]\\
            H_{d-1}(K(\Gamma)) \arrow[r,"i_{\gamma*}"]              &   H_{d-1}(K(E \cup \Gamma)) \arrow[r,"j_{\gamma*}"]   &   H_{d-1}(K(E_k \cup \Gamma)).
        \end{tikzcd}
    \end{equation*}
    As $E_k$ is a Reifenberg set, we have $j_* \circ i_* = 0$ on $L$. The second axiom of general coverings implies that the simplicial complexes $K(E \cup \Gamma)$ and $K(E_k \cup \Gamma)$ have the same $d$-simplexes. Hence the $d$-chains of $K(E \cup \Gamma)$ and $K(E_k \cup \Gamma)$ are identical and they induce the same boundaries. We deduce that $j_{\gamma*}$ is injective and then that $i_{\gamma*} = 0$ on $\pi_\Gamma(L)$. Since every open covering $\alpha$ of $E \cup \Gamma$ is refined by such covering $\gamma$, we conclude that $i_*$ is nul on $L$.

    \emph{Step 1.} We fix a relative open covering $\alpha = (\alpha_i)_i$ of $E \cup \Gamma$ and we build a locally finite open sequence $\beta = (\beta_j)_{j \in \N}$ in $\R^n$ such that
    \begin{enumerate}
        \item $\beta$ cover $E \cup \Gamma$ and $((E \cup \Gamma) \cap \beta_j)_j$ is a refinement of $\alpha$;
        \item for every non-empty finite set $S \subset \N$, the intersection of boundaries $\bigcap_S \partial \beta_i$ is included in a finite union of $(n-m)$-spheres, where $m$ is the cardinal of $S$;
        \item for every non-empty finite set $S \subset \N$,
            \begin{equation}
                \bigcap_S \beta_j \ne \emptyset \implies (E \cup \Gamma) \cap \bigcap_S \beta_j \ne \emptyset.
            \end{equation}
    \end{enumerate}
    We work with the closed set $F := E \cup \Gamma$. For all $x \in F$, there exists $i$ such that $x \in \alpha_i$ so there exists an open ball $B$ centred at $x$ such that $F \cap 2B \subset \alpha_i$. We extract a sequence of open ball $(B_j)_{j \in \N}$ covering $F$ such that $(2B_j)_j$ is locally finite in $\R^n$ and $(F \cap 2B_j)_j$ is a refinement of $\alpha$. Next, we build by induction an open sequence $(\beta_j)_{j \in \N}$ such that for all $j$,
    \begin{enumerate}
        \item $F \cap \overline{B_j} \subset \beta_j$ and there exists $i$ such that $F \cap \beta_j \subset \alpha_i$.
        \item for every non-empty set $S \subset \set{1,\ldots,j}$, the intersection of boundaries $\bigcap_S \partial \beta_i$ is included in a finite union of $(n-m)$-spheres, where $m$ is the cardinal of $S$;
        \item for every non-empty set $S \subset \set{1,\ldots,j}$,
            \begin{equation}
                \bigcap_S \beta_i \ne \emptyset \implies F \cap \bigcap_S \beta_i \ne \emptyset.
            \end{equation}
    \end{enumerate}
    We start with $\beta_0 = 2B_0$. Let us assume that $\beta_0,\ldots,\beta_{j-1}$ has been built. For all $x \in F \cap \overline{B_j}$, there exists an open ball $B$ centered at $x$ such that
    \begin{enumerate}
        \item $B \subset 2B_j$;
        \item for every set $S \subset \set{1,\ldots,j-1}$, the intersection of boundaries $\partial B \cap \bigcap_S \partial \beta_i$ is included in a finite union of $(n-m-1)$-spheres, where $m$ is the cardinal of $S$; 
        \item for every set $S \subset \set{1,\ldots,j-1}$ such that $(F \cap \overline{B_j}) \subset \R^n \setminus \bigcap_S \overline{\beta_j}$, we have $\overline{B} \subset \R^n \setminus \bigcap_S \overline{\beta_j}$. Equivalently,
            \begin{equation}
                \overline{B} \cap \bigcap_S \overline{\beta_j} \ne \emptyset \implies (F \cap \overline{B_j}) \cap \bigcap_S \overline{\beta_j} \ne \emptyset.
            \end{equation}
    \end{enumerate}
    The first and third condition mean that the radius of $B$ is small enough. The second conditions holds for almost all radius of $B$ according to Lemma \ref{spheres_positions}. Extract a finite covering of $F \cap \overline{B_j}$ by such balls $B$ and denote $\beta_j$ their union. Then $\beta_j$ solves the next step of the induction.

    \emph{Step 2. We complete the family $\beta$ with an open set $\beta_\infty$ to obtain a covering of one of the $E_k$. We take care not to introduce new $d$-simplexes in the sense of (\ref{simplicial_complexes}).} We want to reduce the problem to the case where for some $k$, $E_k \setminus \bigcup_j \beta_j$ is a $(d-1)$-dimensional grid. Using a Federer--Fleming projection, we are going to project $E_k$ in a $(d-1)$-dimensional grid away from $E \cup \Gamma$. Let $\ell > 0$ and consider a complex $K$ describing a uniform grid of sidelength $(\sqrt{n})^{-1}\ell$ in $\R^n$. Note that
    \begin{equation}
        \R^n = \bigcup \set{A | A \in K} = \bigcup \set{\mathrm{int}(A) | A \in K}
    \end{equation}
    and the cells of $K$ have a diameter $\leq \ell$. We select the cells in which we want to perform the Federer--Fleming projection. Let $B_0$ be an open ball such that for all $k$, $E_k \subset \overline{B_0}$. Let $L$ be the subcomplex of $K$ defined by
    \begin{equation}
        L = \set{A \in K | \exists x \in A,\ x \in \overline{2B_0} \ \text{and} \ \mathrm{d}(x,E \cup \Gamma) \geq 2\ell}.
    \end{equation}
    Since $\R^n$ is covered by the interiors of cells $A \in K$, it is clear that
    \begin{equation}\label{grille1}
        \set{x \in \overline{2B_0} | \mathrm{d}(x,E \cup \Gamma) \geq 2\ell} \subset \bigcup \set{\mathrm{int}(A) | A \in L}.
    \end{equation}
    The set $E \cup \Gamma$ is a closed set included in $\bigcup_j \beta_j$ so the function $x \mapsto \mathrm{d}(x,E \cup \Gamma)$ is positive on $\R^n \setminus \bigcup_j \beta_j$. In particular, it has a positive minimum on the compact set $\overline{2B_0} \setminus \bigcup_j \beta_j$. This minimum does not depend on $\ell$ so we can assume $\ell$ small enough so that for all $x \in \overline{2B_0} \setminus \bigcup_j \beta_j$, $\mathrm{d}(x,E \cup \Gamma) > 4\ell$. By contraposition,
    \begin{equation}\label{grille2}
        \set{x \in \overline{2B_0} | \mathrm{d}(x,E \cup \Gamma) \leq 4\ell} \subset \bigcup_j \beta_j.
    \end{equation}
    Next, we introduce the Federer--Fleming projection of $E_k \cap \abs{L}$ in $L$. First, we justify that $\lim_k \HH^d(E_k \cap \abs{L}) = 0$. By local finitness of $K$, $\abs{L}$ is a closed set of $\R^n$. Since the cells of $K$ have a diameter $\leq \ell$, the definition of $L$ implies that the cells of $L$ cannot meet $E \cup \Gamma$. The set $V = \R^n \setminus \abs{L}$ is open and contains $E \cup \Gamma$ so according to the lemma assumptions,
    \begin{equation}
        \lim_k \HH^d(E_k \cap \abs{L}) = 0.
    \end{equation}
    Let an index $k$ be such that $\HH^d(E_k \cap \abs{L}) < \infty$ (it will be precised later). We apply \cite[Lemma 2.9]{Labourie} or \cite[Propositon 3.1]{DS} and we obtain a continuous map $\phi\colon \abs{L} \to \abs{L}$ such that 
    \begin{enumerate}
        \item for all $A \in L$, $\phi(A) \subset A$;
        \item $\phi(E_k \cap \abs{L}) \subset \abs{L} \setminus \bigcup \set{\mathrm{int}(A) | A \in L,\ \mathrm{dim}(A) > d}$;
        \item for all $A \in L^d$,
            \begin{equation}\label{third_axiom}
                \HH^d(\phi(E_k \cap \abs{L}) \cap A) \leq C \HH^d(E_k \cap \abs{L})
            \end{equation}
    \end{enumerate}
    where $C$ is a positive constant that depends only on $n$. When $k$ is big enough (depending on $\ell$), $\HH^d(E_k \cap \abs{L})$ becomes sufficiently small so that one can perform additional projections in the $d$-dimensional cells of $L$. The second axiom becomes
    \begin{equation}
        \phi(E_k \cap \abs{L}) \subset \abs{L} \setminus \bigcup \set{\mathrm{int}(A) | A \in L,\ \mathrm{dim}(A) \geq d};
    \end{equation}
    and thus
    \begin{multline}\label{phi_image}
        \phi(E_k \cap \abs{L}) \cap \bigcup \set{\mathrm{int}(A) | A \in L} \\\subset \bigcup \set{\mathrm{int}(A) | A \in L,\ \mathrm{dim}(A) \leq d-1}.
    \end{multline}
    The sets $E \cup \Gamma$ and $\abs{L}$ are disjoint and closed so we can extend $\phi$ continuously on $E \cup \Gamma$ by $\phi = \mathrm{id}$. Observe that $\abs{\phi - \mathrm{id}} \leq \ell$ because $\phi$ preserves the cells of $L$. We can extend $\phi$ continuously on $\R^n$ in such a way that $\abs{\phi - \mathrm{id}} \leq \ell$. Now, let us show that
    \begin{equation}\label{Ek_image}
        \phi(E_k) \subset \bigcup \set{A | A \in L,\ \mathrm{dim}(A) \leq d-1} \cup \bigcup_{j \in \N} \beta_j.
    \end{equation}
    Remember that $E_k \subset \overline{B_0}$. We assume that $\ell$ is less than the radius of $B_0$ whence $\phi(E_k) \subset \overline{2B_0}$. For $x \in E_k$, we distinguish two cases. If $\mathrm{d}(x,E \cup \Gamma) \leq 3\ell$, then $\mathrm{d}(\phi(x),E \cup \Gamma) \leq 4\ell$ and we have $\phi(x) \in \bigcup_j \beta_j$ by (\ref{grille2}). If $\mathrm{d}(x,E \cup \Gamma) \geq 3\ell$, then we have both $\mathrm{d}(x,E \cup \Gamma) \geq 2\ell$ and $\mathrm{d}(\phi(x),E \cup \Gamma) \geq 2\ell$ so (\ref{grille1}) shows that $x$ and $\phi(x)$ belongs to $\bigcup \set{\mathrm{int}(A) | A \in L}$. Then by (\ref{phi_image}), we have $\phi(x) \in \bigcup \set{A | A \in L,\ \mathrm{dim}(A) \leq d-1}$.

    We are all set to introduce the set
    \begin{equation}
        \beta_\infty = \R^n \setminus (E \cup \Gamma \cup \bigcup_{\abs{S} = d} \bigcap_S \overline{\beta_j}).
    \end{equation}
    First, we justify that the set $\beta_\infty$ is open. It suffices to show that the family $\left(\bigcap_S \overline{\beta_j}\right)_{\abs{S}=d}$ is locally finite in $\R^n$. In step 1, we have built the family $(\beta_j)_{j \in \N}$ in such a way that it is locally finite in $\R^n$ so for all $x \in \R^n$, there exists an open set $U$ containing $x$ such that
    \begin{equation}
        S_0 := \set{j \in \N | U \cap \beta_j \ne \emptyset}
    \end{equation}
    is finite. Let $S$ be any subset of $\N$ with cardinal $d$ such that $U$ meets $\bigcap_S \overline{\beta_j}$. In fact, $U$ meets also $\bigcap_S \beta_j$ because $U$ is open so $S \subset S_0$. We deduce that there exists only a finite number of sets $S \subset \N$ of cardinal $d$ such that $U$ meets $\bigcap_S \overline{\beta_j} \ne \emptyset$. We have proved that $\left(\bigcap_S \overline{\beta_j}\right)_{\abs{S}=d}$ is locally finite in $\R^n$. Next, we justify that $\beta_\infty$ does not add new $d$-simplexes. The definition of $\beta_\infty$ shows that for all $S \subset \N$ of cardinal $d$,
    \begin{equation}\label{gamma_simplex}
        \beta_\infty \cap \bigcap_S \beta_j = \emptyset
    \end{equation}
    so for all set $S \subset \N \cup \set{\infty}$ of cardinal $d+1$, the condition $\bigcap_S \beta_j \ne \emptyset$ implies $S \subset \N$. Finally, we would like to have the covering
    \begin{equation}\label{Ek_image2}
        \phi(E_k) \subset \beta_{\infty} \cup \bigcup_{j \in \N} \beta_j.
    \end{equation}
This is where the projection of $E_k$ in a $(d-1)$-dimensional grid outside $\bigcup_j \beta_j$ is useful. According to (\ref{Ek_image}), the condition (\ref{Ek_image2}) holds if
\begin{equation}
        \bigcup \set{A | A \in L,\ \mathrm{dim}(A) \leq d-1} \setminus \beta_\infty \subset \bigcup_j \beta_j.
\end{equation}
As $E \cup \Gamma \subset \bigcup_j \beta_j$, this amounts to say that for all $S \subset \N$ of cardinal $d$, 
    \begin{equation}
        \bigcup \set{A | A \in L,\ \mathrm{dim}(A) \leq d-1} \cap \bigcap_S \overline{\beta_j} \subset \bigcup_j \beta_j.
    \end{equation}
    We are going to see that for after a suitable translation of $K$, the set
    \begin{equation}
        \abs*{K^{d-1}} := \bigcup \set{A | A \in K,\ \mathrm{dim}(A) \leq d-1}
    \end{equation}
    is disjoint from $\bigcap_S \partial \beta_j$ for all $S \subset \N$ of cardinal $d$. Fix such a set $S$. The set $\bigcap_S \partial \beta_j$ is included in a finite union of $(n-d)$-spheres so for all $(d-1)$-linear plane $P$,
    \begin{equation}
        \LL^n \left(\bigcap_S \partial \beta_j + P\right) = 0
    \end{equation}
    and in particular,
    \begin{equation}
        \LL^n \left(\bigcap_S \partial \beta_j + (-\abs*{K^{d-1}}) \right) = 0.
    \end{equation}
    This means that for almost every $x \in \R^n$, $x + \abs*{K^{d-1}}$ is disjoint from $\bigcap_S \partial \beta_j$. There are only a countable number of sets $S \subset \N$ of cardinal $d$ so we can find $x$ such that this is true for all of them. We replace $K$ by $x + K$ so that (\ref{Ek_image2}) holds.

    \emph{Step 3. We build the general covering $\gamma$.} We consider the domain $V_\gamma = \N \cup \set{\infty}$ and for $j \in V_\gamma$, we consider the open set $\gamma_j = \phi^{-1}(\beta_j)$. Remember that $\phi = \mathrm{id}$ on $E \cup \Gamma$ so for all $j \in V_\gamma$,
    \begin{equation}
        (E \cup \Gamma) \cap \gamma_j = (E \cup \Gamma) \cap \beta_j
    \end{equation}
    and in particular, $(E \cup \Gamma) \cap \gamma_{\infty} = \emptyset$. By definition of $(\beta_j)_{j \in \N}$, the family $((E \cup \Gamma) \cap \gamma_j)_{V_\gamma}$ is a refinement of $\alpha$. The family $\gamma$ also covers $E_k \cup \Gamma$ because
    \begin{equation}
        \phi(E_k \cup \Gamma) \subset \phi(E_k) \cup \Gamma \subset \bigcup_{j \in V_\gamma} \beta_j.
    \end{equation}
    Finally, we show that for all $S \subset V_\gamma$ of cardinal $d+1$,
    \begin{equation}
        \bigcap_S \gamma_j \ne \emptyset \implies (E \cup \Gamma) \cap \bigcap_S \gamma_j \ne \emptyset.
    \end{equation}
    If $\bigcap_S \gamma_j \ne \emptyset$, then $\bigcap_S \beta_j \ne \emptyset$ and thus $S \subset \N$ because $\beta_\infty$ does not add new $d$-simplexes. By definition of $(\beta_j)_{j \in \N}$, this implies $(E \cup \Gamma) \cap \bigcap_S \beta_j \ne \emptyset$ or equivalently,
    \begin{equation}
        (E \cup \Gamma) \cap \bigcap_S \gamma_j \ne \emptyset
    \end{equation}
    because $\gamma_j$ coincides with $\beta_j$ on $E \cup \Gamma$.
\end{proof}

\section{Existence of Plateau solutions}
We solve two formulations of the Reifenberg Plateau problem. In the first one, we work in $X = \R^n$ and minimize $\II^d(E)$ among Reifenberg competitors $E$. In the second one, we work in $X = \R^n \setminus \Gamma$ (that is, away from the boundary) and minimize $\II^d(E \setminus \Gamma)$ among Reifenberg competitors $E$. In this case, we do not require regularity on the boundary.

\subsection{Direct method}
We are going to recall the direct method \cite[Corollary 4.1]{Labourie}. This is the same direct method as in \cite{I1} but this version allows any ambient space (even containing the boundary).

The ambient space is an open set $X$ of $\R^n$. See \cite[Definition 1.5]{Labourie} for the definition of a \emph{minimal set} in $X$. See \cite[Definition 1.8]{Labourie} for the definition of a \emph{Whitney subset} of $X$. It includes smooth compact manifolds imbedded in $X$.
\begin{prop}[Direct method]\label{cor_direct}%_direct_method
    Fix a Whitney subset $\Gamma$ of $X$. Fix an admissible energy $\II$ in $X$. Fix $\mathcal{C}$ a class of closed subsets of $X$ such that
    \begin{equation}
        m := \inf \set{\II(E) | E \in \mathcal{C}} < \infty
    \end{equation}
    and assume that for all $E \in \mathcal{C}$, for all sliding deformations $f$ of $E$ in $X$,
    \begin{equation}
        m \leq \II(f(E)).
    \end{equation}
    Let $(E_k)$ be a minimizing sequence for $\II$ in $\mathcal{C}$. Up to a subsequence, there exists a coral\footnote{A set $E \subset X$ is coral in $X$ if $E$ is the support of $\HH^d \mres E$ in $X$. Equivalently, $E$ is closed in $X$ and for all $x \in E$ and for all $r > 0$, $\HH^d(E \cap B(x,r)) > 0$.} minimal set $E_\infty$ with respect to $\II$ in $X$ such that
    \begin{equation}
        \II \mres E_k \rightharpoonup \II \mres E_\infty.
    \end{equation}
    where the arrow $\rightharpoonup$ denotes the weak convergence of Radon measures in $X$. In particular, $\II(E_\infty) \leq m$.
\end{prop}
    In the works of Reifenberg, the Hausdorff distance limit of a minimizing sequence is a competitor but the area is not lower semicontinuous. With weak limits, the lower semicontinuity follows from the previous proposition. Moreover, we also proved in the last section that the weak limit is a competitor. There is not much work to do.

\subsection{Applications}
\begin{thm}[with the free boundary]\label{rei_sol1}%_rei_sol1
    Fix a Whitney subset $\Gamma$ of $\R^n$ and fix a subgroup $L$ of $H_{d-1}(\Gamma)$. We assume that 
    \begin{equation}
        m := \inf \set{\II(E) | E \ \text{Reifenberg competitor}} < \infty
    \end{equation}
    and that there exists a compact set $C \subset \R^n$ such that
    \begin{equation}
        m = \inf \set{\II(E) | E \ \text{Reifenberg competitor},\ E \subset C}.
    \end{equation}
    Then there exists a Reifenberg competitor $E \subset C$ such that $\II(E) = m$.
\end{thm}
\begin{proof}
    We work in $X = \R^n$ and we consider the class
    \begin{equation}
        \mathcal{C} = \set{E | E \ \text{is a Reifenberg competitor}}.
    \end{equation}
    By Lemma \ref{rei_image}, the class $\mathcal{C}$ is closed under sliding deformations in $\R^n$ so it satisfies the requirement of Proposition \ref{cor_direct}. Let $(E_k) \subset C$ be a minimizing sequence of $\mathcal{C}$. According to Proposition \ref{cor_direct}, there exists a coral set $E_\infty$ of $\R^n$ such that
    \begin{equation}
        \II \mres E_k \rightharpoonup \II \mres E_\infty.
    \end{equation}
    We prove that $E_\infty$ is a Reifenberg competitor. First, we show that $E_\infty$ is a compact subset of $C$. Observe that $\R^n \setminus C$ is an open set and that by lower semicontinuity,
    \begin{equation}
        \II(E_\infty \setminus C) \leq \liminf_k \II(E_k \setminus C) = 0.
    \end{equation}
    This proves that the support of $\II \mres E_\infty$ is included in $C$. As $E_\infty$ is coral, $E_\infty$ is included in $C$ and in particular compact. We are going to apply Lemma \ref{rei_limit} to the set $E_\infty$. For all open set $V$ containing $E_\infty \cup \Gamma$,
    \begin{align}
        \limsup_k \II(E_k \setminus V)  &   = \limsup_k \II(E_k \cap C \setminus V)\\
                                        &   \leq \II(E_\infty \cap C \setminus V)\\
                                        &   \leq 0.
    \end{align}
    We conclude that $E_\infty$ is a Reifenberg competitor. Finally, we show that $\II(E_\infty) = m$. As $E_\infty$ is a Reifenberg competitor, we have of course $\II(E_\infty) \geq m$. The fact that $\II(E_\infty) \leq m$ has already been established in Proposition \ref{cor_direct}.
\end{proof}

The next theorem is equivalent to \cite[Theorem 1.3]{Fang} (which is based on Feuvrier's construction) and \cite[Theorem 3.4]{I5} (which is based on weak limits of minimizing sequences).
\begin{thm}[without the free boundary]\label{rei_sol2}%_rei_sol2
    Fix a closed set $\Gamma$ of $\R^n$ and a subgroup $L$ of $H_{d-1}(\Gamma)$. We assume that 
    \begin{equation}
        m := \inf \set{\II(E \setminus \Gamma) | E \ \text{Reifenberg competitor}} < \infty
    \end{equation}
    and that there exists a compact set $C \subset \R^n$ such that
    \begin{equation}
        m = \inf \set{\II(E \setminus \Gamma) | E \ \text{Reifenberg competitor},\ E \subset C}.
    \end{equation}
    Then there exists a Reifenberg competitor $E \subset C$ such that $\II(E \setminus \Gamma) = m$.
\end{thm}
\begin{rmk}
    If $\Gamma$ is compact and $\II(\Gamma) < \infty$, this amounts to minimizing $\II(E)$ among Reifenberg competitors containing $\Gamma$.
\end{rmk}
\begin{proof}
    We work in $X = \R^n \setminus \Gamma$ (away from the boundary) and we consider the class
    \begin{equation}
        \mathcal{C} = \set{E \setminus \Gamma | E \ \text{is a Reifenberg competitor}}.
    \end{equation}
    By Lemma \ref{rei_image}, the class $\mathcal{C}$ is closed under sliding deformations in $X$ (the boundary is empty in $X$) so it satisfies the requirement of Proposition \ref{cor_direct}. Let $(E_k) \subset C$ be a sequence of Reifenberg competitor such that $(E_k \setminus \Gamma)$ is a minimizing sequence of $\mathcal{C}$. According to Proposition \ref{cor_direct}, there exists a coral set $S_\infty$ of $X$ such that
    \begin{equation}
        \II \mres (E_k \setminus \Gamma) \rightharpoonup \II \mres S_\infty \ \text{in $X$}.
    \end{equation}
    We prove that there exists a Reifenberg competitor $E_\infty \subset C$ such that $S_\infty = E_\infty \setminus \Gamma$. First, we justify that $S_\infty \subset C$. Observe that $X \setminus C$ is an open set of $X$ and that by lower semicontinuity,
    \begin{equation}
        \II(S_\infty \setminus C) \leq \liminf_k \II((E_k \setminus \Gamma) \setminus C) = 0.
    \end{equation}
    As a consequence, the support of $\II \mres S_\infty$ in $X$ is included in $C$. As $S_\infty$ is coral in $X$, $S_\infty$ is included in $C$. Now, let
    \begin{equation}
        E_\infty = (S_\infty \cup \Gamma) \cap C.
    \end{equation}
    The set $S_\infty$ is closed in $X$ so $S_\infty \cup \Gamma$ is closed in $\R^n$ so $E_\infty$ is a compact subset of $C$. Note also that $E_\infty \setminus \Gamma = S_\infty$. We are going to apply Lemma \ref{rei_limit} to the set $E_\infty$. For all open set $V$ containing $E_\infty \cup \Gamma$,
    \begin{align}
        \limsup_k \II(E_k \setminus V)  &   = \limsup_k \II(E_k \cap C \setminus V)\\
                                        &   = \limsup_k \II((E_k \setminus \Gamma) \cap C \setminus V)\\
                                        &   \leq \II(E_\infty \cap C \setminus V)\\
                                        &   \leq 0.
    \end{align}
    Thus, $E_\infty$ is a Reifenberg competitor and $S_\infty = E_\infty \setminus \Gamma \in \mathcal{C}$. Finally, we show that $\II(S_\infty) = m$. As $S_\infty \in \mathcal{C}$, we have of course $\II(S_\infty) \geq m$. The fact that $\II(S_\infty) \leq m$ has already been established in Proposition \ref{cor_direct}.
\end{proof}

\textbf{Acknowledgement:} I would like to thank Guy David for his warm and helpful discussions. I thank the anonymous reviewer which has improved this paper.

\textsc{Université Paris-Saclay, CNRS, Laboratoire de mathématiques d'Orsay, 91405, Orsay, France.}
\end{document}